\documentclass[11pt,oneside,a4paper]{article}
\usepackage{amsmath}
\usepackage{amsfonts}
\usepackage{amssymb}
\usepackage{amsthm}
\usepackage{dsfont}
\usepackage{enumitem}
\usepackage{hyperref}
\usepackage{indentfirst}
\usepackage{mathrsfs}
\usepackage[top=1in,bottom=1in,left=1.25in,right=1.25in]{geometry}
\newtheorem{definition}{Definition}[section]
\newtheorem{lemma}{Lemma}[section]
\newtheorem{theorem}{Theorem}[section]
\newtheorem{proposition}{Proposition}[section]
\newtheorem{remark}{Remark}[section]
\allowdisplaybreaks
\begin{document}

\renewcommand{\baselinestretch}{1.2}
\renewcommand{\arraystretch}{1.0}

\title{\bf Multiplier Hopf coquasigroup: Definition and Coactions}
\author
{
    \textbf{Tao Yang} \footnote{Corresponding author.
        College of Science, Nanjing Agricultural University, Nanjing 210095, China. E-mail: tao.yang@njau.edu.cn}
}
\maketitle

\begin{center}
    \begin{minipage}{12.cm}

        \textbf{Abstract}
        This paper uses Galois maps to give a definition of generalized multiplier Hopf coquasigroups, 
        and give a sufficient and necessary condition for a multiplier bialgebra to be a regular multiplier Hopf coquasigroup.
        Then coactions and Yetter-Drinfeld quasimodules of regular multiplier Hopf coquasigroups are also considered.
        \\

        {\bf Keywords} multiplier Hopf coquasigroup, Galois map, coaction, Yetter-Drinfeld module
        \\

        {\bf Mathematics Subject Classification}   16T05 $\cdot$ 16T99

    \end{minipage}
\end{center}
\normalsize

\section{Introduction}
\def\theequation{\thesection.\arabic{equation}}
\setcounter{equation}{0}

 Multiplier Hopf algebras were introduced by Van Daele in \cite{V94}, generalizing Hopf algebras to a non-unital case. 
 An important feature of this setting is the extension of the dual of finite-dimensional Hopf algebras, within one category \cite{V98}.
 Motivated by multiplier Hopf algebra theory, the paper \cite{Y22} used the non-degenerate faithful integrals to construct the integral dual of an infinite-dimensional Hopf quasigroup, 
 and got the multiplier Hopf coquasigroup structure and bidual theorem.  
 Unlike Hopf coquasigroups in \cite{K,KM}, the underlying algebra of multiplier Hopf coquasigroup $A$ need not have an identity in general,
 and the comultiplication $\Delta$ does not map $A$ into $A \otimes A$ but rather into its multiplier algebra $M(A \otimes A)$.
 Multiplier Hopf coquasigroups weaken the coassociativity of multiplier Hopf algebra, 
 and unify the notions of multiplier Hopf algebras and Hopf quasigroups.
 
 As we know, Galois maps play an important role in the definition of multiplier Hopf algebra. Then there is a natural question:
 Could the definition of multiplier Hopf coquasigroup be restated by Galois maps?
 This is the main motivation, and this paper give an positive answer.

 This paper is organized as follows. 
 In section 2, we introduce some basic notions:  multiplier algebras, completed module and multiplier Hopf coquasigroups, which will be used in the following sections.

 In section 3, starting from the Galois maps, we restate the definition of (regular) multiplier Hopf ($*$-)coqusigroup, and get the equivalent definition of regular multiplier Hopf coqusigroup.

 In section 4, as in multiplier Hopf algebra case \cite{YW11}, we use completed module to define the coactions of multiplier Hopf coquasigroups and Yetter-Drinfeld quasimodules, 
 which could be regarded as a natural generalization of Yetter-Drinfeld modules. The four kinds of Yetter-Drinfeld quasimodule categories are equivalent.

\section{Preliminaries}
\def\theequation{\thesection.\arabic{equation}}
\setcounter{equation}{0}

 Throughout this paper, all linear spaces we considered are over a fixed field $k$ (e.g., the complex number field $\mathds{C}$).

\subsection{Multiplier algebras and completed modules}

 Let $A$ be an (associative) algebra over $k$ , with or without identity, but we do require that
 the product, seen as a bilinear form, is non-degenerated. This means that, whenever $a\in A$ and $ab=0$ for all $b\in A$
 or $ba=0$ for all $b\in A$, we must have that $a=0$.

 Recall from \cite{V94,V98} that $M(A)$ is characterized as the largest algebra with identity containing $A$ as an essential two-sided ideal.
 In particularly, we still have that, whenever $a\in M(A)$ and $ab=0$ for all $b\in A$ or $ba=0$ for all $b\in A$, again $a=0$.
 Furthermore, we consider the tensor algebra $A\otimes A$. It is still non-degenerated and we have its multiplier algebra $M(A\otimes A)$.
 There are natural imbeddings
 $$A\otimes A \subseteq M(A)\otimes M(A) \subseteq M(A\otimes A).$$
 In generally, when $A$ has no identity, these two inclusions are strict.
 If $A$ already has an identity $1_{A}$, the product is obviously non-degenerate
 and $M(A)=A$ and $M(A\otimes A) = A\otimes A$.

 Let $A$ and $B$ be non-degenerate algebras, if homomorphism $f: A\longrightarrow M(B)$ is non-degenerated
 (i.e., $f(A)B=B$ and $Bf(A)=B$),
 then it has a unique extension to a homomorphism $M(A)\longrightarrow M(B)$, we also denote it $f$.
 \\

 Let $X$ be a vector space over $k$. 
 Suppose $X$ is a left $A$-module with the module structure map $\cdot : A\otimes X\longrightarrow X$. 
 We will always assume non-degenerate module actions, that is, $x=0$ if $x\in X$ and $a\cdot x=0$ for all $a\in A$.
 If the module is unital (i.e., $A\cdot X=X$), then we can get an extension of the module structure to $M(A)$, 
 this means that we can define $f\cdot x$, where $f\in M(A)$ and $x\in X$.
 In fact, since $x\in X=A\cdot X$, then $x= \sum_{i} a_{i}\cdot x_{i}$, $f\cdot v=\sum_{i} (fa_{i})\cdot x_{i}$.
 In this setting, we can easily get $1_{M(A)}\cdot x=x$.

 Consider a non-degenerate algebra $A$ and a left $A$-module $X$ such that the module action is non-degenerate. 
 Let us recall from \cite{V08,YW11} the completed (or extended) module of $X$.
 
 Denote by $Y$ the space of linear maps $\rho: A\rightarrow X$ satisfying $\rho(aa')=a\cdot \rho(a')$ for all $a, a'\in A$. 
 Then $Y$ becomes a left $A$-module if we define $a\cdot \rho$ for $a\in A$ and $\rho \in Y$ by $(a\cdot \rho)(a')=\rho(a'a)=a'\cdot \rho(a)$.
 Define $\rho_{x} \in Y$ by $\rho_{x}(a)=a\cdot x$ when $a\in A$. Then $X$ becomes a submodule of $Y$. And we have $A\cdot Y\subseteq X$, and if $A\cdot X=X$ then
 $A\cdot Y=X$. Since $A^{2}=A$, $Y$ is still non-degenerate. If $A$ has a unit, then $Y=X$, in the other case,
 mostly $Y$ is strictly bigger than $X$. We can also do the same as before for right modules as well.

 Let $X$ be a non-degenerate $A$-bimodule. 
 Denotes by $Z$ the space of pair $(\lambda, \rho)$ of linear maps from $A$ to $X$ satisfying $a\cdot \lambda(a')=\rho(a)\cdot a'$ for all $a, a' \in A$. 
 From the non-degeneracy, it follows that $\rho(aa')=a\cdot \rho(a')$ and $\lambda(aa')=\lambda(a)\cdot a'$ for all $a, a' \in A$. 
 Also $\rho$ is completely determined by $\lambda$ and vice versa. 
 We can consider $Z$ as the intersection of two extensions of $X$ (as a left and a right modules). 
 Then $Z$ becomes an $A$-bimodule, if we define $az$ and $za$ for $a\in A$ and $z=(\lambda, \rho) \in Z$ by $az = (a\lambda(\cdot), \rho(\cdot a))$ and $za = (\lambda(a \cdot), \rho(\cdot) a)$. 
 If we define $(\lambda_{x}, \rho_{x})$ for $x\in X$ by $\lambda_{x}(a)=x\cdot a$ and $\rho_{x}(a)=a\cdot x$, we get $X$ as a submodule of $Z$.
 We say that $Z$ is a {\it completed (or extended) module} of $A$, and denote as $M_{0}(X)$.

\subsection{Multiplier Hopf coquasigroups}

A \emph{multiplier Hopf coquasigroup} introduced in \cite{Y22} is a nondegenerate associative algebra $A$ equipped with algebra homomorphisms $\Delta: A\longrightarrow M(A\otimes A)$(coproduct),
$\varepsilon: A\longrightarrow k$ (counit), and a linear map $S: A \longrightarrow A$ (antipode) such that
\begin{enumerate}[nosep, label=(\arabic*)]
    \item  $T_{1}(a\otimes b)=\Delta(a)(1\otimes b)$ and $T_{2}(a\otimes b)=(a\otimes 1)\Delta(b)$ belong to $A\otimes A$ for any $a, b\in A$.
    \item The counit satisfies $(\varepsilon\otimes id)T_{1}(a\otimes b) = ab = (id\otimes \varepsilon)T_{2}(a\otimes b)$.
    \item $S$ is antimultiplicative and anticomultiplicative such that for any $a, b\in A$
          \begin{eqnarray}
              && (m\otimes \iota)(S\otimes \Delta)\Delta(a) = 1_{M(A)}\otimes a = (m\otimes \iota)(\iota \otimes S\otimes \iota)(\iota \otimes \Delta)\Delta(a), \label{2.1}\\
              && (\iota \otimes m)(\Delta \otimes S)\Delta(a) = a \otimes 1_{M(A)}  = (\iota\otimes m)(\iota \otimes S\otimes \iota)(\Delta \otimes \iota)\Delta(a). \label{2.2}
          \end{eqnarray}
\end{enumerate}
Multiplier Hopf coquasigroup $(A, \Delta)$ is called \emph{regular}, if the antipode $S$ is bijective. $T_{1}$ and $T_{2}$ in the above are usually called \emph{Galois maps}.

Multiplier Hopf coquasigroup weakens the coassociativity of multiplier Hopf algebra. And if multiplier Hopf coquasigroup $(A, \Delta)$ is coassociative, then $(A, \Delta)$ is a multiplier Hopf algebra.
From the perspective of Hopf quasigroup and coquasigroup, multiplier Hopf coquasigroup gives an answer to the dual of infinite-dimensional Hopf quasigroup, generalizes Hopf coquasigroup to a non-unital case, and retains many of the properties of Hopf coquasigroup.

 Recall from Definition 2.2 in \cite{B10} a $k$-linear map $\phi$ is \emph{almost left $A$-colinear} if $\phi = (\iota \otimes \varepsilon \otimes \iota) (\iota \otimes \phi) (\Delta \otimes \iota)$,
while $\phi$ is \emph{almost right $A$-colinear} if $\phi = (\iota \otimes \varepsilon \otimes \iota) (\phi \otimes \iota ) (\iota \otimes \Delta)$.
In the following proposition, we show that for a multiplier Hopf coquasigroup $(A, \Delta)$ introduced in \cite{Y22}, the Galois map $T_{2}$ is almost right $A$-colinear, and  the $T_{1}$ is almost left $A$-colinear.
\begin{proposition}
    For a multiplier Hopf coquasigroup $(A, \Delta)$, the Galois map $T_{1}$ (respectively $T_{2}$) is almost left (respectively right) $A$-colinear, and has a left (respectively right) $A$-colinear inverse.
\end{proposition}
\begin{proof}
    We only check $T_{1}$ here and $T_{2}$ is similar.
    For any $c\in A$,
    \begin{eqnarray*}
        && (c\otimes 1)(\iota \otimes \varepsilon \otimes \iota) (\iota \otimes T_{1}) (\Delta \otimes \iota)(a\otimes b) \\
        &=& (\iota \otimes \varepsilon \otimes \iota) (\iota \otimes T_{1}) \big((c\otimes 1\otimes 1)(\Delta \otimes \iota)(a\otimes b)\big) \\
        &=& (\iota \otimes \varepsilon \otimes \iota) (\iota \otimes T_{1}) (ca_{(1)}\otimes a_{(2)}\otimes b) \\
        &=& (\iota \otimes \varepsilon \otimes \iota) (ca_{(1)}\otimes a_{(2)(1)}\otimes a_{(2)(2)}b) \\
        &=& ca_{(1)}\otimes \varepsilon (a_{(2)(1)}) a_{(2)(2)}b \\
        &=& ca_{(1)}\otimes a_{(2)}b \\
        &=& (c\otimes 1) T_{1} (a\otimes b).
    \end{eqnarray*}
    Therefore $T_{1} = (\iota \otimes \varepsilon \otimes \iota) (\iota \otimes T_{1}) (\Delta \otimes \iota)$, i.e., almost left $A$-colinearity holds.

    As in \cite{Y22} $T_{1}$ is invertible and $T_{1}^{-1} (a\otimes b) = a_{(1)}\otimes S(a_{(2)})b$. Indeed,
   \begin{eqnarray*}
        && T_{1}\big(T_{1}^{-1} (a\otimes b)\big) = T_{1}\big( a_{(1)}\otimes S(a_{(2)})b \big)
        = a_{(1)(1)}\otimes a_{(1)(2)} S(a_{(2)})b \stackrel{(\ref{2.2})}{=} a\otimes b, \\
        && T_{1}^{-1}\big(T_{1} (a\otimes b)\big) = T_{1}^{-1}\big( a_{(1)}\otimes a_{(2)}b \big)
        = a_{(1)(1)}\otimes S(a_{(1)(2)}) a_{(2)} b \stackrel{(\ref{2.2})}{=} a\otimes b.
    \end{eqnarray*} 
     Next we check $T_{1}^{-1}$ is also left $A$-colinear. Indeed, for $\forall c\in A$,
    \begin{eqnarray*}
        && (c\otimes 1)(\iota \otimes \varepsilon \otimes \iota) (\iota \otimes T_{1}^{-1}) (\Delta \otimes \iota)(a\otimes b) \\
        &=& (\iota \otimes \varepsilon \otimes \iota) (\iota \otimes T_{1}^{-1}) \big((c\otimes 1\otimes 1)(\Delta \otimes \iota)(a\otimes b)\big) \\
        &=& (\iota \otimes \varepsilon \otimes \iota) (\iota \otimes T_{1}^{-1}) (ca_{(1)}\otimes a_{(2)}\otimes b) \\
        &=& (\iota \otimes \varepsilon \otimes \iota) (ca_{(1)}\otimes a_{(2)(1)}\otimes S(a_{(2)(2)})b ) \\
        &=& ca_{(1)}\otimes S(\varepsilon (a_{(2)(1)}) a_{(2)(2)}) b \\
        &=& ca_{(1)}\otimes S(a_{(2)})b \\
        &=& (c\otimes 1) T_{1}^{-1} (a\otimes b).
    \end{eqnarray*}
    Thus, almost left $A$-colinearity of  $T_{1}^{-1}$ holds.
\end{proof}

\section{Galois maps of multiplier Hopf coquasigroups}
\def\theequation{\thesection.\arabic{equation}}
\setcounter{equation}{0}

As we know, Galois maps play a key role in the definition of multiplier Hopf algebra. 
In this section, we try to construct antipode from the Galois maps following the strategy of multiplier Hopf algebra,
and give an equivalent definition of regular multiplier Hopf coquasigroup.

Let $(A, m)$ be a nondegenerate associative algebra with its multiplier algebra $M(A)$. For the sake of narrative convenience, let's first introduce concept of comultiplication.

\begin{definition}
    Let $A$ be an nondegenerate associative algebra, a \emph{comultiplication} on $A$ is an algebra homomorphisms $\Delta: A\longrightarrow M(A\otimes A)$ such that $T_{1}(a\otimes b)=\Delta(a)(1\otimes b)$ and $T_{2}(a\otimes b)=(a\otimes 1)\Delta(b)$ belong to $A\otimes A$ for all $a, b\in A$.
   
    A \emph{counit} $\varepsilon$ on $A$ is a  $k$-linear map $\varepsilon: A\longrightarrow k$ such that 
    \begin{eqnarray}
   (\varepsilon \otimes \iota)T_{1}(a\otimes b)  = ab, \qquad
    (\iota \otimes \varepsilon)T_{2}(a\otimes b)  = ab. \label{3.1}
    \end{eqnarray}
\end{definition}

\begin{remark}
 (1) The comultiplication $\Delta$ is not necessarily \emph{coassociative} in the sense that 
 \begin{eqnarray*}
 (T_{2}\otimes \iota)(\iota\otimes T_{1}) = (\iota\otimes T_{1}) (T_{2}\otimes \iota).
 \end{eqnarray*}

 (2) We use the usual ‘Sweedler’ notation for coalgebras $\Delta(a) = a_{(1)}\otimes a_{(2)}$, etc. The equation (\ref{3.1}) implies that 
 \begin{eqnarray*}
   &&  (\varepsilon \otimes \iota)T_{1}(a\otimes b)  =  (\varepsilon \otimes \iota)(\Delta(a)(1\otimes b))  =  (\varepsilon \otimes \iota)(a_{(1)}\otimes a_{(2)}b)  = \varepsilon(a_{(1)}) a_{(2)} b = ab, \\
   && (\iota \otimes \varepsilon)T_{2}(a\otimes b)   =  (\iota \otimes \varepsilon)((a\otimes 1)\Delta(b))  =  (\iota \otimes \varepsilon)(ab_{(1)}\otimes b_{(2)})  = ab_{(1)} \varepsilon(b_{(2)}) = ab
 \end{eqnarray*} 
 for all $b\in A$, that is $(\varepsilon \otimes \iota)\Delta(a)= a = (\iota \otimes \varepsilon)\Delta(a)$.
  
  (3) $\varepsilon$ is a algebra homomorphism. This means
  \begin{eqnarray}
  \varepsilon(ab) = \varepsilon(a) \varepsilon(b). 
  \end{eqnarray}
  Indeed, by $(\ref{3.1})$ we have $(\iota \otimes \varepsilon)T_{2}(a\otimes bc) = (\iota \otimes \varepsilon)\big((a\otimes 1) \Delta(bc)\big) = abc$. Then
  \begin{eqnarray*}
  (\iota \otimes \varepsilon)\big((a\otimes 1) \Delta(b)\Delta(c)\big) = abc = (\iota \otimes \varepsilon)\big((a\otimes 1) \Delta(b)\big) c.
 \end{eqnarray*} 
 By surjective of $T_{2}$, 
 \begin{eqnarray*}
  (\iota \otimes \varepsilon)\big((a\otimes b)\Delta(c)\big) 
  &=& (\iota \otimes \varepsilon)\big((a\otimes b)\big) c = \varepsilon(b)ac \\
  &=&  \varepsilon(b) (\iota \otimes \varepsilon)\big((a\otimes 1) \Delta(c)\big).
 \end{eqnarray*} 
 Again by surjective of $T_{2}$, we have $(\iota \otimes \varepsilon)(a\otimes bc) = \varepsilon(b) (\iota \otimes \varepsilon)(a\otimes c)$.
 
 (4) If Galois maps $T_{1}, T_{2}$ are bijective, replacing $a\otimes b$ in (\ref{3.1}) by $T_{1}^{-1}(a\otimes b)$, we can also get
  \begin{eqnarray}
  \varepsilon (a) b = (m \circ T_{1}^{-1})(a\otimes b) , \qquad
   a \varepsilon (b) =  (m \circ T_{2}^{-1}) (a\otimes b). \label{3.2}
    \end{eqnarray}
\end{remark}

 Let $(A, \Delta)$ be a nondegenerate associative algebra with a (not necessarily coassociative) comultiplication $\Delta$ and a counit $\varepsilon$,
 without confusion we call $(A, \Delta)$  a \emph{multiplier bialgebra}. Now we define generalized multiplier Hopf coquasigroup, which is chosen to distinguish it from the previous definition in \cite{Y22}.

\begin{definition}
     Let $(A, \Delta, \varepsilon)$ be a multiplier bialgebra. 
     $(A, \Delta)$ is called a \emph{generalized multiplier Hopf coquasigroup}, 
     if Galois maps $T_{1}, T_{2}$ are bijective and $T_{1}$ (respectively $T_{2}$)  has a left (respectively right) $A$-colinear inverse.
     
     Generalized multiplier Hopf coquasigroup $(A, \Delta)$ is \emph{regular}, if $(A, \Delta^{cop})$ is also a generalized multiplier Hopf coquasigroup.
\end{definition}

\begin{remark}
 (1) Following Proposition 2.1, multiplier Hopf coquasigroups introduced in \cite{Y22} are generalized multiplier Hopf coquasigroups.
 
 (2) The comultiplication of generalized multiplier Hopf coquasigroup is not necessarily coassociative. 
 This is the main difference between generalized multiplier Hopf coquasigroups and multiplier Hopf algebras.   
 If the comultiplication in $(A, \Delta)$ is coassociative, then generalized multiplier Hopf coquasigroup $(A, \Delta)$ is actually a multiplier Hopf algebra introduced in \cite{V94}. 
\end{remark}
 
  Let $(A, \Delta)$ be a generalized multiplier Hopf coquasigroup. We will construct an antimultiplicative  $S: A \longrightarrow M(A)$ that has the properties of antipode in multiplier Hopf algebras.
 
 \begin{definition}
  Define a map $S: A \longrightarrow L(A)$ by
    \begin{eqnarray}
   S(a)b = (\varepsilon \otimes \iota) T_{1}^{-1} (a\otimes b), \label{3.3}
    \end{eqnarray}
  here $S(a)\in L(A)$ is indeed a left multiplier for all $a\in A$. 
\end{definition}

\begin{lemma}
  If $T_{1}$ has a left $A$-colinear inverse. Then 
  \begin{eqnarray}
   m((\iota\otimes S)\big((c\otimes 1)\Delta(a)\big)(1\otimes b)) = c\varepsilon(a) b. \label{3.4}
    \end{eqnarray}
\end{lemma}
\begin{proof}
  Since $T_{1}^{-1}$ is the left $A$-colinear inverse of $T_{1}$, then for all $a, b,c \in A$,
  \begin{eqnarray*}
  (c\otimes 1)T_{1}^{-1}(a\otimes b) 
  &=& (c\otimes 1)(\iota \otimes \varepsilon \otimes \iota) (\iota \otimes T_{1}^{-1}) (\Delta \otimes \iota)(a\otimes b) \\
  &=& (\iota \otimes \varepsilon \otimes \iota) (\iota \otimes T_{1}^{-1}) ((c\otimes 1)\Delta(a) \otimes b) \\ 
  &=& (\iota \otimes (\varepsilon \otimes \iota)T_{1}^{-1}) (c a_{(1)}\otimes a_{(2)} \otimes b) \\ 
  &=& c a_{(1)}\otimes S(a_{(2)}) b \\ 
  &=& (\iota\otimes S)\big((c\otimes 1)\Delta(a)\big) (1\otimes b).
  \end{eqnarray*}
  If we apply the multiplication $m$ on the above equation, we get
  \begin{eqnarray*}
  m \Big((\iota\otimes S)\big((c\otimes 1)\Delta(a)\big)(1\otimes b) \Big)
  &=&  m \Big( (c\otimes 1)T_{1}^{-1}(a\otimes b) \Big) \\
  &=&  c m \Big(T_{1}^{-1}(a\otimes b) \Big) \\
  &\stackrel{(\ref{3.2})}{=}& c\varepsilon(a) b.
  \end{eqnarray*}
  This complete the proof.
\end{proof}

\begin{lemma}
  $S$ is antimultiplicative, that is $S(ab) = S(b) S(a)$  for all $a, b \in A$.
\end{lemma}
\begin{proof}
  It is similar to the proof of Lemma 4.4 in \cite{V94}. 
  \begin{eqnarray*}
   m((\iota\otimes S)\big((c\otimes 1)\Delta(a)\Delta(b)\big)(1\otimes d)) 
  &=& m((\iota\otimes S)\big((c\otimes 1)\Delta(ab)\big)(1\otimes d)) \\
  &=& c\varepsilon(ab) d = c\varepsilon(a) d \varepsilon(b) \\
  &=& m((\iota\otimes S)\big((c\otimes 1)\Delta(a) \big)(1\otimes d)) \varepsilon(b) ,
  \end{eqnarray*}
  By the surjectivity of $T_{(2)}$, replacing $(c\otimes 1)\Delta(a)$ by $c\otimes a$ we get
  \begin{eqnarray*}
   m((\iota\otimes S)\big((c\otimes a)\Delta(b)\big)(1\otimes d)) 
  &=& m((\iota\otimes S)\big((c\otimes a)\big)(1\otimes d)) \varepsilon(b)  \\
  &=& c S(a)d \varepsilon(b) = c \varepsilon(b) S(a)d \\
  &=&  m((\iota\otimes S)\big((c\otimes 1)\Delta(b)\big)(1\otimes S(a) d)).  
  \end{eqnarray*}
  Again by the surjectivity of $T_{(2)}$, replacing $(c\otimes 1)\Delta(b)$ by $c\otimes b$ we get
  \begin{eqnarray*}
   m((\iota\otimes S)\big((c\otimes ab)\big)(1\otimes d)) 
  =  m((\iota\otimes S)\big(c\otimes b\big)(1\otimes S(a) d)) ,
  \end{eqnarray*}
  i.e., $cS(ab)d = cS(b)S(a)d$.
\end{proof}

\begin{lemma}
 If  $T_{2}$ also has a right $A$-colinear inverse, then $S(a)$  is also a right multiplier for all $a \in A$, i.e. $S(a)\in R(A)$.
\end{lemma}
\begin{proof}
  Define $S': A \longrightarrow R(A)$ by
  \begin{eqnarray}
   a S'(b) = (\iota \otimes\varepsilon) T_{2}^{-1} (a\otimes b), \label{3.5}
  \end{eqnarray}
  Since $T_{2}^{-1}$ is the right $A$-colinear inverse of $T_{1}$, then for all $a, b,c \in A$,
  \begin{eqnarray*}
  T_{2}^{-1}(c\otimes a)(1\otimes b) 
  &=& \Big( (\iota \otimes \varepsilon \otimes \iota) (T_{2}^{-1} \otimes\iota) (\iota\otimes \Delta)(c\otimes a) \Big) (1\otimes b) \\
  &=& \Big( ((\iota \otimes\varepsilon) T_{2}^{-1} \otimes\iota) \big(c\otimes \Delta (a) \big) \Big) (1\otimes b) \\
  &=& ((\iota \otimes\varepsilon) T_{2}^{-1} \otimes\iota) \big(c\otimes \Delta (a)(1\otimes b) \big) \\
  &=& (\iota \otimes\varepsilon) T_{2}^{-1} (c\otimes a_{(1)}) \otimes a_{(2)}b \\
  &=& c S'(a_{(1)}) \otimes a_{(2)}b \\
  &=& (c\otimes 1)(S'\otimes \iota)  \big(\Delta (a)(1\otimes b) \big).
  \end{eqnarray*}
  Then we apply the multiplication $m$ on the above equation, we get
  \begin{eqnarray*}
  m \Big( (c\otimes 1)(S'\otimes \iota)  \big(\Delta (a)(1\otimes b) \big) \Big)
  &=&  m \Big( T_{2}^{-1}(c\otimes a)(1\otimes b)  \Big) \\
  &=&  m \Big(T_{2}^{-1}(c\otimes a) \Big) b \\
  &\stackrel{(\ref{3.2})}{=}& c\varepsilon(a) b.
  \end{eqnarray*}
  We claim that $S = S'$.  Indeed, since $T_{2}$ is bijective and if $a\otimes b = \sum (a_{i}\otimes 1)\Delta(b_{i})$, then
  \begin{eqnarray*}
  a S'(b) &=& \sum a_{i} \varepsilon(b_{i}), \\
  a\otimes S(b)c  &=& (\iota\otimes S)(a\otimes b) (1\otimes c) = \sum (\iota\otimes S)\big((a_{i}\otimes 1)\Delta(b_{i})\big) (1\otimes c),
  \end{eqnarray*}
  applying the multiplication to the last equation, we get 
  \begin{eqnarray*}
  a S(b)c 
  &=& m(\sum (\iota\otimes S)\big((a_{i}\otimes 1)\Delta(b_{i})\big) (1\otimes c) ) \\
  &\stackrel{(\ref{3.4})}{=}& \sum a_{i} \varepsilon(b_{i}) c \\
  &=& a S'(b) c.
  \end{eqnarray*}
  This follows $S(b) = S'(b)$ for any $b\in A$.
\end{proof}

\begin{remark}
 (1) These lemmas actually show that $S: A \longrightarrow M(A)$ is an anti-homomorphism satisfying $(\ref{3.4})$ and
 \begin{eqnarray}
  m \Big( (c\otimes 1)(S\otimes \iota)  \big(\Delta (a)(1\otimes b) \big) \Big)
  = c\varepsilon(a) b.
  \end{eqnarray}
  And Galois maps 
  \begin{eqnarray*}
  T_{1}^{-1}(a\otimes b)  &=&  a_{(1)}\otimes S(a_{(2)}) b, \\ 
  T_{2}^{-1}(c\otimes a)  &=&  c S(a_{(1)}) \otimes a_{(2)}.
  \end{eqnarray*}
  
  (2) $T_{2} \circ T_{2}^{-1} = \iota = T_{2}^{-1} \circ T_{2}$ implies the equation (\ref{2.1}), while  $T_{1} \circ T_{1}^{-1} = \iota = T_{1}^{-1} \circ T_{1}$ implies the equation (\ref{2.2}).
  
  Indeed, take $T_{1}$ for as an example,
  \begin{eqnarray*}
  && a\otimes b 
  = (T_{1} \circ T_{1}^{-1} )(a\otimes b) 
  = T_{1}( a_{(1)}\otimes S(a_{(2)}) b) =  a_{(1)(1)}\otimes a_{(1)(2)}S(a_{(2)}) b, \\
  && a\otimes b 
  = (T_{1}^{-1} \circ T_{1} )(a\otimes b) 
  = T_{1}^{-1} ( a_{(1)}\otimes a_{(2)} b) =  a_{(1)(1)}\otimes S(a_{(1)(2)}) a_{(2)} b.
  \end{eqnarray*}
  
  (3) $\varepsilon\circ S = \varepsilon$.
  
   In fact,
   \begin{eqnarray*}
   \varepsilon(S(a)b) 
   &=& \varepsilon \big( (\varepsilon \otimes \iota) T_{1}^{-1} (a\otimes b) \big) 
   = (\varepsilon \otimes \varepsilon) T_{1}^{-1} (a\otimes b) \\
   &\stackrel{(*)}{=}& \varepsilon \big( m \circ T_{1}^{-1} (a\otimes b) \big) 
   \stackrel{(\ref{3.2})}{=} \varepsilon (\varepsilon (a) b) \\
   &=& \varepsilon (a) \varepsilon (b),
    \end{eqnarray*}
    where $(*)$ holds because $\varepsilon$ is a algebra homomorphism.
\end{remark}

Follow the above lemmas, we can get one of the main results.

 \begin{theorem}
 If $(A, \Delta)$ be a generalized multiplier Hopf coquasigroup with an identity, then $(A, \Delta)$ is a Hopf coquasigroup.
 \end{theorem}

 In the following we discuss the regularity of $(A, \Delta)$.
 Let $(A, \Delta)$ be a regular generalized multiplier Hopf coquasigroup. Then $(A, \Delta^{cop})$ is again a generalized multiplier Hopf coquasigroup. 
 Let $\varepsilon^{cop}$ and $S^{cop}$ be the associate counit and antipode.

\begin{proposition}
  (1) $\varepsilon^{cop} = \varepsilon$;
  
  (2) $S(A)\subseteq A$, $S^{cop}(A)\subseteq A$ and $S^{cop} = S^{-1}$. 
\end{proposition}
\begin{proof}
    The proof is the same as in Lemma 5.1 and Proposition 5.2 of the paper \cite{V94}.
    
    (1) If rewrite the formula $(\ref{3.1})$ for $\varepsilon^{cop}$, we find that 
    \begin{eqnarray*}
    ab 
    &=& (\iota \otimes \varepsilon^{cop})T_{2}(a\otimes b)  
    = (\iota \otimes \varepsilon^{cop})\big((a\otimes 1)\Delta^{cop}(b)) \big) \\
    &=& (\varepsilon^{cop} \otimes\iota)\big((1\otimes a)\Delta(b)) \big) ,
    \end{eqnarray*}
    on the other hand, $ (\varepsilon \otimes\iota)\big((1\otimes a)\Delta(b)) \big)c = a (\varepsilon \otimes\iota)\big(\Delta(b)(1\otimes c)) \big) =abc$, i.e.,
    \begin{eqnarray*}
    ab     
    &=& (\varepsilon \otimes\iota)\big((1\otimes a)\Delta(b) \big).
    \end{eqnarray*}
 The surjectivity of $T_{2}$ and the flip map implies that $a\otimes b \longrightarrow (1\otimes a)\Delta(b)$ is bijective, hence $\varepsilon^{cop} = \varepsilon$.
 
 (2) If $a\otimes b = T_{1}(\sum a_{i}\otimes b_{i}) = \sum \Delta( a_{i})(1\otimes b_{i})$, then by $(\ref{3.4})$ $S(a)b = (\varepsilon \otimes \iota) T_{1}^{-1} (a\otimes b) = \sum \varepsilon( a_{i}) b_{i}$.
 Then apply the flip map, 
 \begin{eqnarray*}
  b\otimes a =  \sum \Delta^{cop}( a_{i})(b_{i}\otimes 1),\quad  b\otimes ac =  \sum \Delta^{cop}( a_{i})(b_{i}\otimes c).
 \end{eqnarray*}
 If we apply $S^{cop}\otimes \iota$ and multiply with $d\otimes 1$, 
 \begin{eqnarray*}
  dS^{cop}(b)\otimes ac 
  &=&  \sum (d\otimes 1)(S^{cop}\otimes \iota) \big( \Delta^{cop}( a_{i})(b_{i}\otimes c) \big)  \\
  &=&  \sum (d\otimes 1)(S^{cop}(b_{i})\otimes 1)(S^{cop}\otimes \iota) \big( \Delta^{cop}( a_{i})(1 \otimes c) \big),  
 \end{eqnarray*}
 and take the multiplication, we get
 \begin{eqnarray*}
  d S^{cop}(b)ac 
  &=&   \sum d S^{cop}(b_{i})\varepsilon( a_{i}) c =  dS^{cop} \big(\sum\varepsilon( a_{i})(b_{i})\big) c =  dS^{cop} \big( S(a)b \big) c.  
 \end{eqnarray*}
 This shows that $$ S^{cop}(b)a = S^{cop} \big( S(a)b \big).$$
 From Definition 3.3, the elements of the form $ S(a)b$ span $A$. Then the above formula implies $S^{cop}(A)\subseteq A$, and $cS^{cop}(b)a = cS^{cop}(b)S^{cop} \big( S(a) \big)$.
 Therefore $S^{cop} \big( S(a) \big) = a$, i.e., $S^{cop} = S^{-1}$. 
\end{proof}
 
 If generalized multiplier Hopf coquasigroup $(A, \Delta)$ is commutative or cocommutative, i.e., $\Delta^{cop}=\Delta$, then $(A, \Delta)$ is automatically regular and $S^{cop} = S$, $S^{2} = \iota$.
 Following the above proposition, we give an characterization of regularity.
 
 \begin{theorem}
 Let $(A, \Delta)$ be a generalized multiplier Hopf coquasigroup. $(A, \Delta)$ is regular if and only if  $S$ is invertible.
 \end{theorem}
 
 Now we will prove  $S$ is anticomultiplicative, that is $\Delta S(a) = (S\otimes S)\Delta^{cop}(a)$  for all $a \in A$.
 
 \begin{lemma}
 Suppose that $a, b, a_{i}, b_{i} \in A$, then 
 \begin{enumerate}[nosep, label=(\arabic*)]
    \item $a\otimes S(b) = \sum \Delta( a_{i})(1\otimes b_{i}) $ \quad iff \quad $(1\otimes b)\Delta(a) = \sum a_{i} \otimes S^{-1} (b_{i})  $.
    \item $S(b)\otimes a = \sum (b_{i}\otimes 1)\Delta( a_{i}) $ \quad iff \quad $\Delta(a)(b\otimes 1) = \sum S^{-1} (b_{i}) \otimes a_{i}$.
    \item The following are equivalent:
    \begin{eqnarray*}
    &(i)& \Delta (a)(1\otimes b) =  \sum \Delta( a_{i})(b_{i}\otimes 1), \\
    &(ii)& a\otimes S^{-1}(b) = \sum (a_{i}\otimes 1)\Delta( b_{i}), \\ 
    &(iii)& (1\otimes a)\Delta (b) = \sum S(b_{i}) \otimes a_{i}.     
   \end{eqnarray*}
\end{enumerate}
 \end{lemma}

 \begin{proposition}
  $S$ is anticomultiplicative, that is $\Delta S(a) = (S\otimes S)\Delta^{cop}(a)$.
 \end{proposition}

 \begin{theorem}
 Let $(A, \Delta)$ be a generalized multiplier Hopf coquasigroup. $(A, \Delta)$ is regular if and only if $(A, \Delta)$ is a regular multiplier Hopf coquasigroup introduced in \cite{Y22}.
 \end{theorem}

 Following Proposition 2.1 and Theorem \thesection.3, we obtain the sufficient and necessary conditions for multiplier bialgebra $(A, \Delta)$ to be a regular multiplier Hopf coquasigroup. 
 And we could restate the definition of regular multiplier Hopf coquasigroup.

 \begin{definition}
  Let $(A, \Delta)$ be a a nondegenerate associative algebra equipped with a (not necessarily coassociative) comultiplication $\Delta: A\longrightarrow M(A\otimes A)$  and a $k$-linear map $\varepsilon: A\longrightarrow k$. 
  $(A, \Delta)$ is called a \emph{generalized multiplier Hopf coquasigroup}, if the following conditions hold.
  \begin{enumerate}[nosep, label=(\arabic*)]
    \item  $T_{1}(a\otimes b)=\Delta(a)(1\otimes b)$ and $T_{2}(a\otimes b)=(a\otimes 1)\Delta(b)$ belong to $A\otimes A$ for any $a, b\in A$.
    \item The counit satisfies $(\varepsilon\otimes id)T_{1}(a\otimes b) = ab = (id\otimes \varepsilon)T_{2}(a\otimes b)$.
    \item $T_{1}$ (respectively $T_{2}$)  has a left (respectively right) $A$-colinear inverse, i.e., for all $a, b, c \in A$,
    \begin{eqnarray*}
    && (\iota \otimes \varepsilon \otimes \iota) (\iota \otimes T_{1}^{-1}) (\Delta \otimes \iota)(a\otimes b) = T_{1}^{-1} (a\otimes b), \\
    && (\iota \otimes \varepsilon \otimes \iota) (T_{2}^{-1} \otimes \iota) (\iota \otimes \Delta)(c\otimes a) = T_{2}^{-1}(c\otimes a).      
  \end{eqnarray*}

  Generalized multiplier Hopf coquasigroup  $(A, \Delta)$ is called a \emph{regular multiplier Hopf coquasigroup}, if $(A, \Delta^{cop})$ is again a generalized multiplier Hopf coquasigroup.
\end{enumerate}
\end{definition}

 Let us consider the case of a $*$-algebra over field $\mathds{C}$. In this situation regulairy is automatic. And the cunit and the antipode keep the involution in the following sense,
 \begin{proposition}
 If $(A, \Delta)$ is a generalized multiplier Hopf $*$-coquasigroup, then
 
  (1) $\varepsilon$ is a $*$-homomorphism, $\varepsilon(a^*) = \varepsilon(a)^-$, where $\varepsilon(a)^-$ is the conjugate complex number of $\varepsilon(a)$.
  
  (2) $S\big(S(a)^* \big)^* = a$ for all $a\in A$. 
\end{proposition}

\section{Coactions of regular multiplier Hopf coquasigroups}
\def\theequation{\thesection.\arabic{equation}}
\setcounter{equation}{0}

 Bases on the complete modules, we first introduce the coaction of regular multiplier Hopf coquasigroups, and Yetter-Drinfeld quasicomodules,
 which generalizes Yetter-Drinfeld quasicomodules of Hopf coquasigroups \cite{AFGS15} and multiplier Hopf algebra \cite{YW11}. 
 \\
 
 Let $(A, \Delta)$ be a multiplier Hopf coquasigroup, then $A$ has local units in the sense of Proposition 4.5 in \cite{Y22}. 
 Let $V$ be a vector space and $X=V\otimes A$. 
 Consider a left and a right action of $A$ on $X$ by for all $a, a'\in A$ and $v\in V$,
 \begin{eqnarray*}
 a\cdot (v\otimes a')= v\otimes aa', \qquad
 (v\otimes a')\cdot a= v\otimes a'a.
 \end{eqnarray*}
 The completed module we get here is denoted as $M_{0}(V\otimes A)$.

 \begin{definition}
 Let $A$ be a regular multiplier Hopf coquasigroup and $V$ a vector space.
  A left quasi-coaction of $A$ on $V$ is an injective linear map $\Gamma_{l}: V \rightarrow M_{0}(A\otimes V), v \mapsto v_{(-1)} \otimes v_{(0)} $ satisfying
  for all $a\in A$,
  \begin{eqnarray}
    && (\varepsilon \otimes \iota) \big((a\otimes 1)\Gamma_{l}(v) \big) = \varepsilon(a)v,  \\
    && a v_{(-1)}S(v_{(0)(-1)}) \otimes v_{(0)(0)} = a \otimes v = a S(v_{(-1)})v_{(0)(-1)} \otimes v_{(0)(0)}.      
  \end{eqnarray}  
 Here we use the adapted version of the Sweedler notation for this quasi-coaction, and write $(a \otimes 1)\Gamma_{l}(v) = a v_{(-1)}\otimes v_{(0)}$ for all $a\in A$ and $v\in V$.
 Sometimes $V$ is called a left $A$-quasicomodule.

 Given two left $A$-quasicomodules $(V, \Gamma_{l})$ and $(V, \Gamma_{l}')$, a linear map $f: V\longrightarrow V'$ is a morphism of left $A$-quasicomodules 
 if $\Gamma_{l}'\circ f = (\iota\otimes f)\circ \Gamma_{l}$.
 \end{definition}

 Similarly, we have a right  $A$-quasicomodule with the following conditions:
  for all $a\in A$,
  \begin{eqnarray}
    && (\iota \otimes \varepsilon) \big(\Gamma_{r}(v)(1\otimes a) \big) = \varepsilon(a)v,  \\
    && v_{(0)(0)} \otimes v_{(0)(1)}S(v_{(1)})a = v \otimes a = v_{(0)(0)}\otimes S(v_{(0)(1)})v_{(1)}a.      
  \end{eqnarray}  
  We will denote by ${}^{A}\mathcal{QM}$ the category of left $A$-quasicomodules.
  Moreover if $V$ is a usual left $A$-comodule, then $V$ is a left $A$-quasicomodule.
  Obviously, all left $A$-comodules with the obvious morphisms form a subcategory of ${}^{A}\mathcal{QM}$,
  this subcategory will be denoted by ${}^{A}\mathcal{M}$.

  In the same manner, we define the categories  ${}^{A}\mathcal{QM}^{A}$ and ${}^{A}\mathcal{M}^{A}$, 
  which satisfy the same compatible condition $a v_{(-1)} \otimes v_{(0)(0)} \otimes v_{(0)(1)}b = a v_{(0)(-1)} \otimes v_{(0)(0)} \otimes v_{(1)} b$ for $v\in V$ and any $a,b \in A$.
 \\
 
 For  multiplier Hopf coquasigroup $A$, the notions of left $A$-modules is exactly the same as for multiplier Hopf algebras \cite{DVZ}, since it only depends on the algebra structure of $A$.
 That is, $(M, \cdot)$ (or $M$ in shot) is a left $A$-module if $M$ is a vector space and $\cdot: A\otimes M \longrightarrow M (a\otimes m \mapsto a\cdot m)$ is a linear map (called the action)
 satisfying the module conditions: the action is unital, i.e., $A\cdot M = M$, and $b\cdot (a\cdot m) = (ba)\cdot m$ for all $a, b\in A$ and $m\in M$. 

 We shall denote by  ${}_{A}\mathcal{M}$ the category of left $A$-modules. 
 Similarly we denite by  $\mathcal{M}_{A}$ the category of all right $A$-modules and ${}_{A}\mathcal{M}_{A}$ the category of $A$-$A$-bimodules.
 \\
 
 In the following, we will introduce Yetter-Drinfeld quasimodules over a regular multiplier Hopf coquasigroup as in \cite{GW20}. 
 
 \begin{definition}
 Let $A$ be a regular multiplier Hopf coquasigroup and $V$ a vector space.
 $V$ is called a left-left Yetter-Drinfeld quasimodule over $A$, if $V$ is a unital left $A$-module with the action $\cdot$, and a left $A$-quasicomodule with the coaction $\Gamma$ satisfy the following equations: 
 for any $a\in A$ and $v\in V$
  \begin{eqnarray}
    && b \big(a_{(1)} \cdot v\big)_{(-1)} a_{(2)} \otimes \big(a_{(1)} \cdot v\big)_{(0)}  = b a_{(1)}v_{(-1)} \otimes a_{(2)} \cdot v_{(0)} \quad \mbox{for all} \quad b\in A, \label{4.5} \\
    && a_{(1)}\cdot v \otimes a_{(2)(1)} \otimes a_{(2)(2)} b = a_{(1)(1)}\cdot v \otimes a_{(1)(2)} \otimes a_{(2)}b \quad \mbox{for all} \quad b\in A, \label{4.6} \\
    && ba_{(1)} \otimes a_{(2)(1)}\cdot v \otimes a_{(2)(2)}c = ba_{(1)(1)} \otimes a_{(1)(2)}\cdot v \otimes a_{(2)}c \quad \mbox{for all} \quad b, c\in A.   
  \end{eqnarray}  
  \end{definition}

 It is not hard to check that the equation (\ref{4.5}) is equivalent to
 \begin{eqnarray}
    && b \big(a \cdot v\big)_{(-1)}  \otimes \big(a \cdot v\big)_{(0)}  = b a_{(1)(1)}v_{(-1)} S(a_{(2)}) \otimes a_{(1)(2)} \cdot v_{(0)} \quad \mbox{for all} \quad b\in A.   
  \end{eqnarray}

 Given two left-left Yetter-Drinfeld quasimodules $V$ and $V'$ over $A$, a linear map $f: V\longrightarrow V'$ is a morphism of left-left Yetter-Drinfeld quasimodules, 
 if $f$ is a morphism of left $A$-modules and left $A$-quasicomodules. We shall denote ${}^{A}_{A}\mathcal{YDQ}$ the category of left-left Yetter-Drinfeld quasimodules over $A$.
 
 Moreover if $V$ is a usual left $A$-comodule, then we say $V$ is a left-left Yetter-Drinfeld module over $A$.
 Obviously, all left-left Yetter-Drinfeld modules with the obvious morphisms form a subcategory of ${}^{A}_{A}\mathcal{YDQ}$,
  this subcategory will be denoted by ${}^{A}_{A}\mathcal{YD}$.

 By paring left or right module structures with left or right quasicomodule structures, there are three other kinds of Yetter-Drinfeld module categories as follows.

 \begin{enumerate}[nosep, label=(\arabic*)]
    \item  The right-left Yetter-Drinfeld quasimodule category ${}^{A}\mathcal{YDQ}_{A}$.    
    $(V, \cdot, \Gamma)\in {}^{A}\mathcal{YDQ}_{A}$ if 
    
     $\bullet$ $(V, \cdot)$ is a right unital $A$-module;
     
     $\bullet$ $(V, \Gamma)$ is a left $A$-quasicomodule;
     
     $\bullet$ the module and quasimodule structures satisfy the following compatible conditions: for any $a, b, c\in A$ and $v\in V$,
     \begin{eqnarray}
    && b a_{(2)} \big( v \cdot a_{(1)}\big)_{(-1)} \otimes \big( v \cdot a_{(1)}\big)_{(0)}   = b v_{(-1)}a_{(1)} \otimes v_{(0)}\cdot a_{(2)},  \\
    && v\cdot a_{(1)} \otimes a_{(2)(1)} \otimes a_{(2)(2)} b = v\cdot a_{(1)(1)} \otimes a_{(1)(2)} \otimes a_{(2)}b, \\
    && ba_{(1)} \otimes v\cdot a_{(2)(1)} \otimes a_{(2)(2)}c = ba_{(1)(1)} \otimes v\cdot a_{(1)(2)} \otimes a_{(2)}c.   
  \end{eqnarray}  
  
    \item  The left-right Yetter-Drinfeld quasimodule category ${}_{A}\mathcal{YDQ}^{A}$.    
    $(V, \cdot, \Gamma)\in {}^{A}\mathcal{YDQ}_{A}$ if 
    
     $\bullet$ $(V, \cdot)$ is a left unital $A$-module;
     
     $\bullet$ $(V, \Gamma)$ is a right $A$-quasicomodule;
     
     $\bullet$ the module and quasimodule structures satisfy the following compatible conditions: for any $a, b, c\in A$ and $v\in V$,
     \begin{eqnarray}
    && \big(a_{(2)} \cdot v\big)_{(0)} \otimes \big(a_{(2)} \cdot v\big)_{(1)} a_{(1)} b  =a_{(1)}\cdot v_{(0)} \otimes a_{(2)} v_{(1)} b,  \\
    && b a_{(1)} \otimes a_{(2)(1)} \otimes a_{(2)(2)}\cdot v = b a_{(1)(1)} \otimes a_{(1)(2)} \otimes a_{(2)}\cdot v, \\
    && b a_{(1)} \otimes a_{(2)(1)}\cdot v \otimes a_{(2)(2)}c = b a_{(1)(1)} \otimes a_{(1)(2)}\cdot v \otimes a_{(2)}c.   
  \end{eqnarray}  
  
    \item  The right-right Yetter-Drinfeld quasimodule category $\mathcal{YDQ}^{A}_{A}$.    
    $(V, \cdot, \Gamma)\in {}^{A}\mathcal{YDQ}_{A}$ if 
    
     $\bullet$ $(V, \cdot)$ is a right unital $A$-module;
     
     $\bullet$ $(V, \Gamma)$ is a right $A$-quasicomodule;
     
     $\bullet$ the module and quasimodule structures satisfy the following compatible conditions: for any $a, b, c\in A$ and $v\in V$,
     \begin{eqnarray}
    && \big(v \cdot a_{(2)}\big)_{(0)} \otimes b a_{(1)} \big(v \cdot a_{(2)}\big)_{(1)}  = v_{(0)}\cdot a_{(1)} \otimes b v_{(1)} a_{(2)},  \\
    && b a_{(1)} \otimes a_{(2)(1)} \otimes v \cdot a_{(2)(2)} = b a_{(1)(1)} \otimes a_{(1)(2)} \otimes v\cdot a_{(2)}, \\
    && b a_{(1)} \otimes v\cdot a_{(2)(1)} \otimes a_{(2)(2)}c = b a_{(1)(1)} \otimes v \cdot a_{(1)(2)} \otimes a_{(2)}c.   
  \end{eqnarray}  
\end{enumerate}

 \begin{theorem}
 Let $(A, \Delta)$ be a regular multiplier Hopf coquasigroup, then 
 $${}^{A}_{A}\mathcal{YDQ} \simeq {}_{A} \mathcal{YDQ}^{A} \simeq {}^{A}\mathcal{YDQ}_{A} \simeq \mathcal{YDQ}^{A}_{A}.$$
 \end{theorem}

\section*{Acknowledgements}

The work was partially supported by the China Postdoctoral Science Foundation (No. 2019M651764)
and National Natural Science Foundation of China (No. 11601231).


\end {document}